\documentclass{amsart}
\usepackage{amssymb, amsmath, amsthm, graphics, comment, xspace, enumerate, xcolor}
\newtheorem{thm}{Theorem}[section]

\newtheorem{lem}[thm]{Lemma}
\newtheorem{prop}[thm]{Proposition}
\theoremstyle{definition}
\newtheorem{defn}[thm]{Definition}

\theoremstyle{remark}
\newtheorem{rem}[thm]{Remark}
\numberwithin{equation}{section}

\begin{document}

\title[Doss almost periodic functions, Besicovitch-Doss almost...]{Doss almost periodic 
functions, Besicovitch-Doss almost periodic functions and convolution products}


\author{Marko Kosti\' c}
\address{Faculty of Technical Sciences,
University of Novi Sad,
Trg D. Obradovi\' ca 6, 21125 Novi Sad, Serbia}
\email{marco.s@verat.net}

\begin{abstract}
In the paper under review, we analyze the invariance of Doss almost periodicity and Besicovitch-Doss almost periodicity under the actions of convolution products. 
We thus continue our recent research studies \cite{fedorov-novi} and \cite{NSJOM-besik} by investigating the case in which the solution operator family $(R(t))_{t>0}$ under our consideration has special growth rates
at zero and infinity. In contrast to \cite{NSJOM-besik}, the results obtained in this paper can be incorporated in the qualitative analysis of solutions to abstract (degenerate) inhomogeneous fractional differential equations in Banach spaces.
\end{abstract}



\maketitle

\section{Introduction and preliminaries}\label{cure}

Let $(X,\|\cdot\|)$ and $(Y,\|\cdot\|)_{Y}$ be two non-trivial complex Banach spaces, and let $L(X,Y)$ be the space consisting of all linear continuous operators from $X$ into $Y;$ $L(X)\equiv L(X,X).$  By $\|\cdot \|_{L(X,Y)}$ we denote the norm in $L(X,Y).$

The results obtained in \cite{NSJOM-besik}, concerning the invariance of Besicovitch-Doss almost periodicity under the actions of convolution products, are applicable in the case that the solution operator family $(R(t))_{t>0}\subseteq L(X,Y)$ satisfies the estimate 
\begin{align}\label{integr}
\int^{\infty}_{0}(1+t)\|R(t)\|_{L(X,Y)}\, dt<\infty.
\end{align}
This, in particular, holds if there exists a finite constant $M>0$ such that 
 the following holds:
\begin{align*}
\|R(t)\|_{L(X,Y)}\leq Me^{-ct}t^{\beta -1},\ t>0\mbox{ for some finite constants }c>0,\ \beta \in (0,1].
\end{align*}
Therefore, the results from \cite{NSJOM-besik} can be applied to a large class of abstract (degenerate) inhomogeneous differential equations of first order in Banach spaces.

On the other hand, in the theory of abstract (degenerate) fractional differential equations in Banach spaces, the estimate
\begin{align}\label{isti-ee}
\|R(t)\|_{L(X,Y)}\leq M\frac{t^{\beta -1}}{1+t^{\gamma}},\ t>0\mbox{ for some finite constants }\gamma>1,\ \beta \in (0,1],\ M>0
\end{align}
plays a crucial role. In the case that $\beta-\gamma \geq -1,$ which naturally appears in applications, the estimate  \eqref{integr} does not hold so that the results of \cite{NSJOM-besik} cannot be applied; see \cite{fedorov-novi}, \cite{knjigaho}-\cite{nova-mono} and references cited therein for more details about the subject.

The main purpose of this paper is to analyze the invariance of Doss $p$-almost periodicity and Besicovitch-Doss $p$-almost periodicity under the actions of infinite convolution product 
\begin{align}\label{profa}
G(t)\equiv  t\mapsto \int_{-\infty}^{t}R(t-s)g(s)\, ds,\ t\in {\mathbb R}
\end{align}
and finite convolution product 
\begin{align}\label{hakala}
H(t)\equiv \int^{t}_{0}R(t-s)[g(s)+q(s)]\, ds,\ t\geq 0
\end{align}
where $1\leq p<\infty,$ $(R(t))_{t>0}$ satisfies the estimate \eqref{isti-ee}, $g(\cdot)$ is Doss $p$-almost periodic or Besicovitch-Doss $p$-almost periodic, and $q(\cdot)$ is vanishing in time, in a certain sense. 

The organization of paper is briefly described as follows. After giving some necessary facts about fractional calculus and types of fractional derivatives used in the paper, we analyze Doss almost periodic functions and Besicovitch-Doss almost periodic functions in Subsection \ref{bice-bolje}. Our main contributions are given in Section \ref{for-i to} (Theorem \ref{brzo-kuso}, Theorem \ref{brzo-kuso1} and Proposition \ref{brzo-kusowas}, Proposition \ref{brzo-kusowas1}); in Subsection \ref{prcko-duo}, we present some applications of our abstract theoretical results established.

Fractional calculus and fractional differential equations are rapidly growing fields of research, due to their invaluable importance in modeling real world phenomena appearing in many fields of science and engineering, such
as astrophysics, electronics, diffusion, chemistry, biology,
aerodynamics and thermodynamics. For more details, see \cite{kilbas}-\cite{knjigaho} and references cited therein. In this paper, we use 
the Weyl-Liouville fractional derivatives and Caputo fractional derivatives. The Weyl-Liouville fractional derivative $D_{t,+}^{\gamma}u(t)$ of order $\gamma  \in (0,1)$ is defined for those continuous functions
$u : {\mathbb R} \rightarrow X$
such that $t\mapsto \int_{-\infty}^{t}g_{1-\gamma}(t-s)u(s)\, ds,$ $t\in {\mathbb R}$ is a well-defined continuously differentiable mapping, by
$$
D_{t,+}^{\gamma}u(t):=\frac{d}{dt}\int_{-\infty}^{t}g_{1-\gamma}(t-s)u(s)\, ds,\quad t\in {\mathbb R}.
$$
For further information about Weyl-Liouville fractional derivatives, we refer the reader to \cite{relaxation-peng}.
If $\alpha >0$ and $m=\lceil \alpha \rceil,$ then 
the Caputo fractional derivative\index{fractional derivatives!Caputo}
${\mathbf D}_{t}^{\alpha}u(t)$ is defined for those functions $u\in
C^{m-1}([0,\infty) : X)$ such that $g_{m-\alpha} \ast
(u-\sum_{k=0}^{m-1}u_{k}g_{k+1}) \in C^{m}([0,\infty) : X),$
by
$$
{\mathbf
D}_{t}^{\alpha}u(t)=\frac{d^{m}}{dt^{m}}\Biggl[g_{m-\alpha}
\ast \Biggl(u-\sum_{k=0}^{m-1}u_{k}g_{k+1}\Biggl)\Biggr].
$$
For further information about Caputo fractional derivatives, we refer the reader to \cite{knjigaho}. Here, $g_{\zeta}(t):=t^{\zeta -1}/\Gamma(\zeta),$ where $\Gamma(\cdot)$ denotes the Euler Gamma function ($\zeta>0$), and $\lceil \alpha \rceil:=\inf\{k\in {\mathbb Z} : k\geq \alpha\}.$

Set $(+\infty)^{a}:=+\infty$ for any finite number $a>0.$ We will use the following elementary lemma:

\begin{lem}\label{element}
Suppose that $1\leq p<\infty$ and $\varphi : {\mathbb R} \rightarrow [0,\infty)$ is a non-negative function. Then we have
$$
\limsup_{s\rightarrow +\infty}\Bigl[\varphi (s)^{1/p}\Bigr]=\Biggl[\limsup_{s\rightarrow +\infty}\varphi (s)\Biggr]^{1/p}.
$$
\end{lem}

\begin{proof}
Clearly, with the common consent introduced above, we have
\begin{align*}
\limsup_{s\rightarrow +\infty}&\Bigl[\varphi (s)^{1/p}\Bigr]=\lim_{s\rightarrow +\infty}\sup_{y\geq s}\Bigl[\varphi(y)^{1/p}\Bigr]
\\ =& \lim_{s\rightarrow +\infty}\Biggl[\sup_{y\geq s}\varphi(y)\Biggr]^{1/p}=\Biggl[ \lim_{s\rightarrow +\infty}\sup_{y\geq s}\varphi (y) \Biggr]^{1/p}=\Biggl[\limsup_{s\rightarrow +\infty}\varphi (s)\Biggr]^{1/p}.
\end{align*}
\end{proof}

\subsection{Doss almost periodic functions and Besicovitch-Doss almost periodic functions}\label{bice-bolje}

Let $1\leq p<\infty ,$ and let $I={\mathbb R}$ or $I=[0,\infty).$ Let us recall that the set $D\subseteq I$ is relatively dense iff for each $\epsilon>0$ 
there exists $l>0$ such that any subinterval of $I$ of length $l$ contains at least one element of set $D$. Following A. S. Besicovitch \cite{besik}, for every function $f\in L_{loc}^{p}({\mathbb R} : X),$ we set
$$
\|f\|_{{\mathcal M}^{p}}:=\limsup_{t\rightarrow +\infty}\Biggl[ \frac{1}{2t}\int^{t}_{-t}\| f(s)\|^{p}\, ds\Biggr]^{1/p};
$$
if $f\in L_{loc}^{p}([0,\infty) : X),$ then we set
$$
\|f\|_{{\mathcal M}^{p}}:=\limsup_{t\rightarrow +\infty}\Biggl[ \frac{1}{t}\int^{t}_{0}\| f(s)\|^{p}\, ds\Biggr]^{1/p}.
$$

It is well known that $
\|\cdot \|_{{\mathcal M}^{p}}$ is a seminorm on the space  ${\mathcal M}^{p} (I : X)$ consisting of those $L_{loc}^{p}(I : X)$-functions $f(\cdot)$ for which $
\|f\|_{{\mathcal M}^{p}}<\infty .$ Put $K_{p}(I : X) := \{ f \in {\mathcal M}^{p}(I : X) : 
\|f\|_{{\mathcal M}^{p}}= 0\} $
and
$
M_{p}(I : X) :={\mathcal M}^{p}(I : X) / K_{p}(I : X).
$
The seminorm $
\|\cdot \|_{{\mathcal M}^{p}}$ on ${\mathcal M}^{p}(I : X)$ induces the norm $
\|\cdot \|_{M^{p}}$ on $M^{p}(I : X),$ under which $M^{p}(I : X)$ becomes a Banach space.
A function $f\in L_{loc}^{p}(I : X)$ is said to be Besicovitch $p$-almost periodic iff there exists a sequence of 
$X$-valued trigonometric polynomials converging to $f(\cdot)$ in $(M^{p}(I : X),\|\cdot \|_{M^{p}}).$ By $B^{p}(I : X)$ we denote 
the
vector space consisting of all Besicovitch  $p$-almost periodic functions $I \rightarrow X.$ 

We will use the following lemma (see \cite[Proposition 2.4]{NSJOM-besik} for the case that $I=[0,\infty)$ and the analysis of R. Doss \cite[p. 478]{doss} for the case that $I={\mathbb R}$): 

\begin{lem}\label{ravi-doss}
Let $1\leq p<\infty ,$ and let $q\in L_{loc}^{p}(I : X).$ Then $\|q(t+\cdot)\|_{{\mathcal M}^{p}}=\|q\|_{{\mathcal M}^{p}}$ for all $t\in I.$
\end{lem}

A function $q\in L_{loc}^{p}(I : X)$ is said to be Besicovitch $p$-vanishing iff $\|q\|_{{\mathcal M}^{p}}=0$ (see \cite[Definition 2.3, Proposition 2.4, Corollary 2.5]{NSJOM-besik} for the case that $I=[0,\infty)$). 

Following the fundamental researches of R. Doss \cite{doss}-\cite{doss5}, established for scalar-valued functions, we have recently introduced the following notion in \cite{NSJOM-besik}: 

\begin{defn}\label{doss6}
Let $1\leq p<\infty ,$ and let $f\in L_{loc}^{p}(I : X).$
Then it is said that $f(\cdot)$ is:
\begin{itemize}
\item[(i)] $B^{p}$-bounded iff $
\|f\|_{{\mathcal M}^{p}}<\infty.$
\item[(ii)] $B^{p}$-continuous iff
\begin{align*}
\lim_{\tau \rightarrow 0}\limsup_{t\rightarrow +\infty}\Biggl[\frac{1}{2t}\int^{t}_{-t}\| f(s+\tau)-f(s)\|^{p}\, ds\Biggr]^{1/p}=0,
\end{align*}
in the case that $I={\mathbb R},$  resp.,
\begin{align*}
\lim_{\tau \rightarrow 0+}\limsup_{t\rightarrow +\infty}\Biggl[\frac{1}{t}\int^{t}_{0}\| f(s+\tau)-f(s)\|^{p}\, ds\Biggr]^{1/p}=0,
\end{align*}
in the case that $I=[0,\infty).$
\item[(iii)] Doss $p$-almost periodic iff, for every $\epsilon >0$, the set of numbers $\tau \in I$ for which
\begin{align}\label{doss-ap}
\limsup_{t\rightarrow +\infty}\Biggl[\frac{1}{2t}\int^{t}_{-t}\| f(s+\tau)-f(s)\|^{p}\, ds\Biggr]^{1/p}<\epsilon,
\end{align}
in the case that $I={\mathbb R},$  resp.,
\begin{align*}
\limsup_{t\rightarrow +\infty}\Biggl[\frac{1}{t}\int^{t}_{0}\| f(s+\tau)-f(s)\|^{p}\, ds\Biggr]^{1/p}<\epsilon,
\end{align*}
in the case that $I=[0,\infty),$ is relatively dense in $I.$
\item[(iv)] Besicovitch-Doss $p$-almost periodic iff (i)-(iii) hold as well as, for every $\lambda \in {\mathbb R}$, we have that
\begin{align}\label{idiotishen}
\lim_{l \rightarrow +\infty}\limsup_{t\rightarrow +\infty}\frac{1}{l}\Biggl[\frac{1}{2t}\int^{t}_{-t}\Biggl\| \Biggl( \int^{x+l}_{x}-\int^{l}_{0}\Biggr) e^{i\lambda s}f(s)\, ds \Biggr \|^{p}\, dx\Biggr]^{1/p}=0,
\end{align}
in the case that $I={\mathbb R},$  resp.,
\begin{align*}
\lim_{l \rightarrow +\infty}\limsup_{t\rightarrow +\infty}\frac{1}{l}\Biggl[\frac{1}{t}\int^{t}_{0}\Biggl\| \Biggl( \int^{x+l}_{x}-\int^{l}_{0}\Biggl) e^{i\lambda s}f(s)\, ds\Biggr \|^{p}\, dx\Biggr]^{1/p}=0,
\end{align*}
in the case that $I=[0,\infty).$
\end{itemize}
\end{defn}

By ${\mathrm B}^{p}(I : X)$
we denote the
class consisting of all Besicovitch-Doss $p$-almost periodic functions $I \rightarrow X.$ Before proceeding further, we would like to note that,  in scalar-valued case $X={\mathbb C}$, R. Doss has proved that ${\mathrm B}^{p}(I : X)=B^{p}(I : X).$ It is still unknown whether this equality holds in vector-valued case (see \cite{NSJOM-besik} for the problem raised). By ${\mathrm D}^{p}(I : X)$
we denote the
class consisting of all Doss $p$-almost periodic functions $I \rightarrow X.$ 

For more details about these classes of generalized almost periodic functions, we refer the reader to the monograph \cite{besik} by A. S. Besicovitch, the survey article \cite{deda} by J. Andres, A. M. Bersani, R. F. Grande, and the forthcoming monograph \cite{nova-mono} by M. Kosti\' c. 

We also need the notion of Stepanov $p$-boundedness. A function $f\in L^{p}_{loc}(I :X)$ is said to be Stepanov $p$-bounded iff
$$
\|f\|_{S^{p}}:=\sup_{t\in I}\Biggl( \int^{t+1}_{t}\|f(s)\|^{p}\, ds\Biggr)^{1/p}<\infty.
$$

\section{Doss almost periodic properties and Besicovitch-Doss almost periodic properties of convolution products}\label{for-i to}

We start this section by stating the following result:

\begin{thm}\label{brzo-kuso}
Let $1/p+1/q=1$ and let $(R(t))_{t>0}\subseteq L(X,Y)$ satisfy \eqref{isti-ee}. Let a function $g : {\mathbb R} \rightarrow X$ be Doss $p$-almost periodic and Stepanov $p$-bounded, and let $q(\beta-1)>-1$ provided that $p>1$, resp. $\beta=1,$ provided that $p=1$. Then the function
$G : {\mathbb R} \rightarrow Y,$ defined through \eqref{profa}, is bounded continuous and Doss $p$-almost periodic. Furthermore, if $g(\cdot)$ is $B^{p}$-continuous, then $G(\cdot)$ is $B^{p}$-continuous, as well.
\end{thm}

\begin{proof}
We primarily analyze the case that $g(\cdot)$ is Doss $p$-almost periodic with $p>1$ and explain certain differences in the proof provided that $p=1;$ the assumption $X=Y$ can be made. Since $g(\cdot)$ is Stepanov $p$-bounded and $q(\beta-1)>-1$, a similar line of reasoning as in the proof of \cite[Theorem 2.1]{fedorov-novi} shows that $G(\cdot)$ is bounded and continuous on the real line.

Now we will prove that $G(\cdot)$ is Doss $p$-almost periodic. 
Let
a number $\epsilon>0$ be fixed. By definition, we can find a real number $l>0$ such that any interval $I\subseteq {\mathbb R}$ of length $l$ contains a point $\tau \in  I$ such that 
\eqref{doss-ap} holds with $f=g$ therein.
Further on, there exists of a positive real number $\zeta >0$ 
satisfying
$
\frac{1}{p}<\zeta <\frac{1}{p}+\gamma-\beta
$
(in the case that $p=1,$ take any number $\zeta \in (1,\gamma)$ and repeat the procedure).
As in the proof of \cite[Theorem 2.1]{fedorov-novi}, we may conclude that
the integral
$$
\int^{0}_{-\infty}\|g(s+t+\tau)-g(s+t)\|^{p} / (1+|s|^{\zeta})^{p} \, ds
$$
converges for any $t\in {\mathbb R}$ as well as that there exists an absolute constant $C>0$ such that
\begin{align*}
\|G(s+\tau)-G(s)\|
\leq C\Biggl[ \int^{0}_{-\infty}\|g(v+s+\tau)-g(v+s)\|^{p} / (1+|v|^{\zeta})^{p} \, dv\Biggr]^{1/p},\ s\in {\mathbb R}.
\end{align*}
Using this estimate, the Fubini theorem and Lemma \ref{element}, we get that
\begin{align*}  
&\limsup_{t\rightarrow +\infty}\Biggl[\frac{1}{2t}\int^{t}_{-t}\| G(s+\tau)-G(s)\|^{p}\, ds\Biggr]^{1/p}
\\ \leq & C\limsup_{t\rightarrow +\infty}
\Biggl[\frac{1}{2t}\int^{t}_{-t} \int^{0}_{-\infty}\|g(v+s+\tau)-g(v+s)\|^{p} / (1+|v|^{\zeta})^{p} \, dv\, ds\Biggr]^{1/p}
\\ = & C\limsup_{t\rightarrow +\infty}
\Biggl[\int^{0}_{-\infty}\frac{1}{2t}\int^{t}_{-t}\|g(v+s+\tau)-g(v+s)\|^{p} / (1+|v|^{\zeta})^{p} \, ds\, dv\Biggr]^{1/p}
\\ = & C\Biggl[\limsup_{t\rightarrow +\infty}
\int^{0}_{-\infty}\frac{1}{2t}\int^{t}_{-t}\|g(v+s+\tau)-g(v+s)\|^{p} / (1+|v|^{\zeta})^{p} \, ds\, dv\Biggr]^{1/p}.
\end{align*}
By the reverse Fatou lemma, the above implies
\begin{align*}  
&\limsup_{t\rightarrow +\infty}\Biggl[\frac{1}{2t}\int^{t}_{-t}\| G(s+\tau)-G(s)\|^{p}\, ds\Biggr]^{1/p}
\\ \leq & C\Biggl[
\int^{0}_{-\infty}\frac{1}{(1+|v|^{\zeta})^{p}}\limsup_{t\rightarrow +\infty}\Biggl\{\frac{1}{2t}\int^{t}_{-t}\|g(v+s+\tau)-g(v+s)\|^{p}\, ds\Biggr\} \, dv\Biggr]^{1/p}.
\end{align*}
Keeping in mind Lemma \ref{ravi-doss}, we get that, for every $v\leq 0,$ 
$$
\limsup_{t\rightarrow +\infty}\frac{1}{2t}\int^{t}_{-t}\|g(v+s+\tau)-g(v+s)\|^{p}\, ds=\limsup_{t\rightarrow +\infty}\frac{1}{2t}\int^{t}_{-t}\|g(s+\tau)-g(s)\|^{p}\, ds,
$$
so that
\begin{align*}  
&\limsup_{t\rightarrow +\infty}\Biggl[\frac{1}{2t}\int^{t}_{-t}\| G(s+\tau)-G(s)\|^{p}\, ds\Biggr]^{1/p}
\\ \leq & C\Biggl[
\int^{0}_{-\infty}\frac{1}{(1+|v|^{\zeta})^{p}}\limsup_{t\rightarrow +\infty}\Biggl\{\frac{1}{2t}\int^{t}_{-t}\|g(s+\tau)-g(s)\|^{p}\, ds\Biggr\} \, dv\Biggr]^{1/p}
\\ = & C\Biggl[\limsup_{t\rightarrow +\infty}\frac{1}{2t}\int^{t}_{-t}\|g(s+\tau)-g(s)\|^{p}\, ds\Biggr]^{1/p}\Biggl[
\int^{0}_{-\infty}\frac{dv}{(1+|v|^{\zeta})^{p}}\Biggr]^{1/p}.
\end{align*}
Applying again Lemma \ref{element}, we finally get
\begin{align*}  
&\limsup_{t\rightarrow +\infty}\Biggl[\frac{1}{2t}\int^{t}_{-t}\| G(s+\tau)-G(s)\|^{p}\, ds\Biggr]^{1/p}
\\ \leq & C\limsup_{t\rightarrow +\infty}\Biggl[\frac{1}{2t}\int^{t}_{-t}\|g(s+\tau)-g(s)\|^{p}\, ds\Biggr]^{1/p}\Biggl[
\int^{0}_{-\infty}\frac{dv}{(1+|v|^{\zeta})^{p}}\Biggr]^{1/p}
\\ \leq & C\epsilon \Biggl[
\int^{0}_{-\infty}\frac{dv}{(1+|v|^{\zeta})^{p}}\Biggr]^{1/p}.
\end{align*}
This shows that $G(\cdot)$ is Doss $p$-almost periodic. Since the above computations holds for each $\tau \in {\mathbb R},$ it is clear that the $B^{p}$-continuity of function $g(\cdot)$ implies that of $G(\cdot),$ as well. The proof of the theorem is thereby complete.
\end{proof}

Concerning the finite convolution product, we can state the following result which immediately follows from Theorem \ref{brzo-kuso} and 
the proof of \cite[Proposition 2.3]{fedorov-novi}:

\begin{prop}\label{brzo-kusowas}
Let $q\in L^{p}_{loc}([0,\infty) : X),$ $1/p+1/q=1$ and let $(R(t))_{t>0}\subseteq L(X,Y)$ satisfy \eqref{isti-ee}. Let a function $g : {\mathbb R} \rightarrow X$ be Doss $p$-almost periodic and Stepanov $p$-bounded, and let $q(\beta-1)>-1$ provided that $p>1$, resp. $\beta=1,$ provided that $p=1$. 
Suppose that the function
\begin{align}\label{preb}
t\mapsto  Q(t)\equiv \int^{t}_{0}R(t-s)q(s)\, ds,\ t\geq 0
\end{align}
belongs to some space ${\mathcal F}_{Y}$ of functions $[0,\infty)\rightarrow Y,$ satisfying that
\begin{align}\label{mravojed}
{\mathcal F}_{Y}+C_{0}([0,\infty) : Y)={\mathcal F}_{Y}.
\end{align}
Then the function $H(\cdot),$ defined through \eqref{hakala},
is continuous and 
belongs to the class ${\mathrm D}^{[0,\infty),p}(Y)+{\mathcal F}_{Y}$, where ${\mathrm D}^{[0,\infty),p}(Y)$ stands for the space of all restictions of $Y$-valued Doss $p$-almost periodic functions from the real line to the interval $[0,\infty).$
\end{prop}

\begin{rem}\label{south}
Suppose that ${\mathcal F}_{Y}=B^{p}_{0}([0,\infty):Y).$ Since the sum of spaces ${\mathrm D}^{p}(I : Y)$ and $B^{p}_{0}(I:Y)$ is ${\mathrm D}^{p}(I : Y)$ (see the proof of \cite[Proposition 2.6]{NSJOM-besik} for the case that $I=[0,\infty)$), the extension of a function from $B^{p}_{0}([0,\infty):Y)$ by zero outside the interval $[0,\infty)$ belongs to the space $B^{p}_{0}({\mathbb R}:Y)$, and \eqref{mravojed} holds, we get that the resulting function $H(\cdot)$ belongs to the space ${\mathrm D}^{[0,\infty),p}(Y),$ as well.
It is also worth noting that it is not clear whether a Doss $p$-almost periodic function defined on $[0,\infty)$ can be extended to a Doss $p$-almost periodic function defined on ${\mathbb R}$.
\end{rem}

Concerning the Besicovitch-Doss $p$-almost periodic functions, the situation is a little bit delicate. It seems that the estimate \eqref{isti-ee} alone is not sufficiently enough to ensure the validity of condition \eqref{idiotishen} for the function $G(\cdot)$ defined through \eqref{profa}. In the following theorem, we impose the additional condition \eqref{montenegro} for the inhomogeneity $g(\cdot)$ under our consideration:

\begin{thm}\label{brzo-kuso1}
Let $1/p+1/q=1$ and let $(R(t))_{t>0}\subseteq L(X,Y)$ satisfy \eqref{isti-ee}. Let a function $g : {\mathbb R} \rightarrow X$ be Besicovitch-Doss $p$-almost periodic and Stepanov $p$-bounded, and let $q(\beta-1)>-1$ provided that $p>1$, resp. $\beta=1,$ provided that $p=1$. 
Suppose, additionally, that for every $\lambda \in {\mathbb R}$ and $\epsilon>0$ there exists a number $l_{0}>0$ such that
\begin{align}\label{montenegro}
\frac{1}{l}\Biggl\| \Biggl( \int^{l}_{0}-\int^{l-v}_{-v}\Biggr)e^{i\lambda s}g(s)\, ds\Biggr\|^{p}<\epsilon,\quad l\geq l_{0},\ v\geq 0.
\end{align}
Then the function
$G : {\mathbb R} \rightarrow Y,$ defined through \eqref{profa}, is bounded continuous and Besicovitch-Doss $p$-almost periodic. 
\end{thm}

\begin{proof}
Let $\lambda \in {\mathbb R}$. By Theorem \ref{brzo-kuso}, we only need to show that the function $G(\cdot)$ satisfies \eqref{idiotishen}. In order to that, choose a positive real number $\zeta >0$ 
satisfying
$
\frac{1}{p}<\zeta <\frac{1}{p}+\gamma-\beta
$
(in the case that $p=1,$ we can take any number $\zeta \in (1,\gamma)$ and repeat the same procedure, again). 
Arguing similarly as in the proof of Theorem \ref{brzo-kuso}, we obtain that:
\begin{align*}
&\frac{1}{l}\Biggl[\frac{1}{2t}\int^{t}_{-t}\Biggl\| \Biggl( \int^{x+l}_{x}-\int^{l}_{0}\Biggr) e^{i\lambda s}G(s)\, ds \Biggr \|^{p}\, dx\Biggr]^{1/p}
\\ \leq  & \frac{M}{l}\Biggl[\frac{1}{2t}\int^{t}_{-t}\Biggl\| \int^{\infty}_{0} \Biggl( \int^{x+l}_{x}-\int^{l}_{0}\Biggr) e^{i\lambda s}g(s-v)\, ds \frac{ v^{\beta-1} dv}{1+v^{\gamma}}\Biggr \|^{p}\, dx\Biggr]^{1/p}
\\ \leq & \frac{M}{l}\Biggl[\frac{1}{2t}\int^{t}_{-t}\Biggl\{ \Biggl\| \frac{1}{1+\cdot^{\zeta}}\Biggl( \int^{x+l}_{x}-\int^{l}_{0}\Biggr) e^{i\lambda s}g(s-\cdot)\, ds\Biggr\|_{L^{p}[0,\infty)} 
\\ \times & \Biggl\|  \frac{ \cdot^{\beta-1}(1+\cdot^{\zeta})}{1+\cdot^{\gamma}}\Biggr\|_{L^{q}[0,\infty)}\Biggr \}^{p}\, dx\Biggr]^{1/p}
\\ \leq & \frac{M}{l}\Biggl[\frac{1}{2t}\int^{t}_{-t}\int^{\infty}_{0}\frac{1}{(1+v^{\zeta})^{p}}\Biggl\| \Biggl( \int^{x+l}_{x}-\int^{l}_{0}\Biggr) e^{i\lambda s}g(s-v)\, ds \Biggr \|^{p}\, dv \, dx\Biggr]^{1/p}
\\ = & \frac{M}{l}\Biggl[\frac{1}{2t}\int^{\infty}_{0}\int^{t}_{-t}\frac{1}{(1+v^{\zeta})^{p}}\Biggl\| \Biggl( \int^{x+l}_{x}-\int^{l}_{0}\Biggr) e^{i\lambda s}g(s-v)\, ds \Biggr \|^{p}\, dx \, dv\Biggr]^{1/p};
\end{align*}
here, $M$ denotes the constant from \eqref{isti-ee}.
Applying the reverse Fatou lemma and Lemma \ref{element}, the above implies:
\begin{align*}
&\limsup_{t\rightarrow +\infty}\frac{1}{l}\Biggl[\frac{1}{2t}\int^{t}_{-t}\Biggl\| \Biggl( \int^{x+l}_{x}-\int^{l}_{0}\Biggr) e^{i\lambda s}G(s)\, ds \Biggr \|^{p}\, dx\Biggr]^{1/p}
\\ \leq & \frac{M}{l}\Biggl[\frac{1}{2t}\int^{\infty}_{0}\int^{t}_{-t}\limsup_{t\rightarrow +\infty}\Biggl\{\frac{1}{2t}\int^{t}_{-t} \Biggl\| \Biggl( \int^{x+l}_{x}-\int^{l}_{0}\Biggr) e^{i\lambda s}g(s-v)\, ds \Biggr \|^{p}\, dx \Biggr\} \, \frac{dv}{(1+v^{\zeta})^{p}}\Biggr]^{1/p}
\\ = & \frac{M}{l}\Biggl[\frac{1}{2t}\int^{\infty}_{0}\int^{t}_{-t}\limsup_{t\rightarrow +\infty}\Biggl\{\frac{1}{2t}\int^{t}_{-t} \Biggl\| \Biggl( \int^{x+l-v}_{x-v}-\int^{l-v}_{-v}\Biggr) e^{i\lambda s}g(s)\, ds \Biggr \|^{p}\, dx \Biggr\} \, \frac{dv}{(1+v^{\zeta})^{p}}\Biggr]^{1/p}.
\end{align*}
Taking into account Lemma \ref{ravi-doss}, we get from the above:
\begin{align*}
&\limsup_{t\rightarrow +\infty}\frac{1}{l}\Biggl[\frac{1}{2t}\int^{t}_{-t}\Biggl\| \Biggl( \int^{x+l}_{x}-\int^{l}_{0}\Biggr) e^{i\lambda s}G(s)\, ds \Biggr \|^{p}\, dx\Biggr]^{1/p}
\\ \leq & 
\frac{M}{l}\Biggl[\frac{1}{2t}\int^{\infty}_{0}\int^{t}_{-t}\limsup_{t\rightarrow +\infty}\Biggl\{\frac{1}{2t}\int^{t}_{-t} \Biggl\| \Biggl( \int^{x+l}_{x}-\int^{l-v}_{-v}\Biggr) e^{i\lambda s}g(s)\, ds \Biggr \|^{p}\, dx \Biggr\} \, \frac{dv}{(1+v^{\zeta})^{p}}\Biggr]^{1/p}
\\ \leq & \frac{2^{\frac{p-1}{p}}M}{l}\Biggl[\frac{1}{2t}\int^{\infty}_{0}\int^{t}_{-t}\limsup_{t\rightarrow +\infty}\Biggl\{\frac{1}{2t}\int^{t}_{-t} \Biggl\| \Biggl( \int^{x+l}_{x}-\int^{l}_{0}\Biggr) e^{i\lambda s}g(s)\, ds \Biggr \|^{p}\, dx \Biggr\} \, \frac{dv}{(1+v^{\zeta})^{p}}\Biggr]^{1/p} 
\\ + &\frac{2^{\frac{p-1}{p}}M}{l}\Biggl[\frac{1}{2t}\int^{\infty}_{0}\int^{t}_{-t}\limsup_{t\rightarrow +\infty}\Biggl\{\frac{1}{2t}\int^{t}_{-t} \Biggl\| \Biggl( \int^{l}_{0}-\int^{l-v}_{-v}\Biggr) e^{i\lambda s}g(s)\, ds \Biggr \|^{p}\, dx \Biggr\} \, \frac{dv}{(1+v^{\zeta})^{p}}\Biggr]^{1/p}.
\end{align*}
Take now any $\epsilon>0;$ then there exists $l_{0}>$ such that \eqref{montenegro} holds. Furthermore, the equation \eqref{idiotishen} holds with $f=g$ so that there exists $l_{1}>0$ such that
\begin{align*}
\frac{1}{l}\Biggl[\frac{1}{2t}\int^{\infty}_{0}\int^{t}_{-t}\limsup_{t\rightarrow +\infty}\Biggl\{\frac{1}{2t}\int^{t}_{-t} \Biggl\| \Biggl( \int^{x+l}_{x}-\int^{l}_{0}\Biggr) e^{i\lambda s}g(s)\, ds \Biggr \|^{p}\, dx \Biggr\} \, \frac{dv}{(1+v^{\zeta})^{p}}\Biggr]^{1/p}<\epsilon
\end{align*}
for all $l\geq l_{1}.$ For $l\geq l_{0},$ the estimate \eqref{montenegro} yields that  
\begin{align*}
\frac{1}{l}\Biggl[\frac{1}{2t}\int^{\infty}_{0}\int^{t}_{-t}\limsup_{t\rightarrow +\infty}\Biggl\{\frac{1}{2t}\int^{t}_{-t} \Biggl\| \Biggl( \int^{l}_{0}-\int^{l-v}_{-v}\Biggr) e^{i\lambda s}g(s)\, ds \Biggr \|^{p}\, dx \Biggr\} \, \frac{dv}{(1+v^{\zeta})^{p}}\Biggr]^{1/p}<\epsilon,
\end{align*}
finishing the proof.
\end{proof}

Now it is quite simple to state the following analogue of Proposition \ref{brzo-kusowas} and Remark \ref{south}:

\begin{prop}\label{brzo-kusowas1}
Let $q\in L^{p}_{loc}([0,\infty) : X),$ $1/p+1/q=1$ and let $(R(t))_{t>0}\subseteq L(X,Y)$ satisfy \eqref{isti-ee}. Let a function $g : {\mathbb R} \rightarrow X$ be Besicovitch-Doss $p$-almost periodic and Stepanov $p$-bounded, and let $q(\beta-1)>-1$ provided that $p>1$, resp. $\beta=1,$ provided that $p=1$. Suppose, additionally, that for every $\lambda \in {\mathbb R}$ and $\epsilon>0$ there exists a number $l_{0}>0$ such that \eqref{montenegro} holds as well as that
the function $Q(\cdot)$ defined through \eqref{preb} 
belongs to some space ${\mathcal F}_{Y}$ of functions $[0,\infty)\rightarrow Y,$ satisfying that \eqref{mravojed} holds.
Then the function $H(\cdot),$ defined through \eqref{hakala},
is continuous and 
belongs to the class ${\mathrm B}^{[0,\infty),p}(Y)+{\mathcal F}_{Y}$, where ${\mathrm B}^{[0,\infty),p}(Y)$ stands for the space of all restictions of $Y$-valued Besicovitch-Doss $p$-almost periodic functions from the real line to the interval $[0,\infty).$
\end{prop}

\begin{rem}\label{south1}
Suppose that ${\mathcal F}_{Y}=B^{p}_{0}([0,\infty):Y).$ Then the resulting function $H(\cdot)$ belongs to the space ${\mathrm B}^{[0,\infty),p}(Y).$ 
It is not clear whether a Besicovitch-Doss $p$-almost periodic function defined on $[0,\infty)$ can be extended to a Besicovitch-Doss $p$-almost periodic function defined on ${\mathbb R},$ as well.
\end{rem}

\subsection[Some  applications]{Some applications}\label{prcko-duo}

Basically, our results can be applied at any place where the variation of parameters formula or some of its generalizations can be applied. 
For example, 
we can incorporate our results in the analysis of abstract fractional inclusions with multivalued linear operators ${\mathcal A}$ satisfying the condition (P) introduced by A. Favini and A. Yagi \cite[p. 47]{faviniyagi}.
Define the resolvent operator families $(S_{\gamma}(t))_{t>0}$ and $(R_{\gamma}(t))_{t>0}$ generated by ${\mathcal A}$ as in \cite{nova-mono} and \cite{fedorov-novi}. Let $x_{0}$ be a point of continuity of $(S_{\gamma}(t))_{t>0},$ i.e., $\lim_{t\rightarrow 0+}S_{\gamma}(t)x_{0}=x_{0}.$
Then there exist two finite constants $M_{1}>0$ and $M_{2}>0$ such that
\begin{align}\label{druga-est}
\bigl\| R_{\gamma}(t) \bigr\|\leq M_{1}t^{\gamma \beta -1},\ t\in (0,1] \ \ \mbox{ and }\ \ \bigl\| R_{\gamma}(t) \bigr\|\leq M_{2}t^{-1-\gamma},\ t\geq 1.
\end{align}

Let $\gamma \in (0,1).$ A continuous function $u: {\mathbb R} \rightarrow X$ is said to be a mild solution of the abstract fractional relaxation inclusion
\begin{align}\label{inkpow}
D_{t,+}^{\gamma}u(t)\in -{\mathcal A}u(t)+f(t),\ t\in {\mathbb R}
\end{align}
iff 
$$
u(t)=\int^{t}_{-\infty}R_{\gamma}(t-s)f(s)\, ds,\ t\in {\mathbb R}.
$$
A mild solution of the abstract fractional relaxation inclusion
\[
\hbox{(DFP)}_{f,\gamma} : \left\{
\begin{array}{l}
{\mathbf D}_{t}^{\gamma}u(t)\in {\mathcal A}u(t)+f(t),\ t> 0,\\
\quad u(0)=x_{0},
\end{array}
\right.
\]
is any function $u\in C([0,\infty) : X)$ satisfying that 
\begin{align*}
u(t)=S_{\gamma}(t)x_{0}+\int^{t}_{0}R_{\gamma}(t-s)f(s)\, ds,\quad t\geq 0.
\end{align*}

Keeping in mind the estimate \eqref{druga-est} and the fact that $\lim_{t\rightarrow +\infty}\|S_{\gamma}(t)\|=0,$ it is clear how we can apply Theorem \ref{brzo-kuso}
and Proposition \ref{brzo-kusowas}
in the analysis of existence and uniqueness of Doss $p$-almost periodic solutions (Besicovitch-Doss $p$-almost periodic solutions) of the abstract fractional inclusions \eqref{inkpow} and 
(DFP)$_{f,\gamma}.$ These results can be simply incorporated in the study of qualitative properties of solutions of the fractional Poisson heat type equations with the Dirichlet Laplacian $\Delta$ in $L^{p}(\Omega),$ where $\Omega$ is an open bounded region in ${\mathbb R}^{n}$ (see \cite{nova-mono} and \cite{fedorov-novi} for more details). Concerning applications of $C$-regularized semigroups (see \cite{knjigah} and references cited therein for more details on the subject), it should be emphasized that we are in aposition to analyze the existence and uniqueness of Doss $p$-almost periodic solutions (Besicovitch-Doss $p$-almost periodic solutions) of 
the initial value problems with constant coefficients
\[
\begin{array}{l}
D_{t,+}^{\gamma}u(t,x)=\sum_{|\alpha|\leq k}a_{\alpha}f^{(\alpha)}(t,x)+f(t,x),\ t\in {\mathbb R},\ x\in \mathbb{R}^n
\end{array}
\]
and
\[\left\{
\begin{array}{l}
{\mathbf D}_{t}^{\gamma}u(t,x)=\sum_{|\alpha|\leq k}a_{\alpha}f^{(\alpha)}(t,x)+f(t,x),\ t\geq 0,\ x\in \mathbb{R}^n,\\
 u(0,x)=u_{0}(x),\quad x\in \mathbb{R}^n
\end{array}
\right.
\]
in the space $L^{p}(\mathbb{R}^n),$ where $1\leq p<\infty $
and some extra assumptions are satisfied.

\bibliographystyle{amsplain}

\end{document}